\documentclass[a4paper,onesided]{article}
\usepackage{amsmath,amsfonts,amssymb,amsthm}

\newcommand{\mr}{{\mathbb R}}

\newcommand{\mn}{{\mathbb N}}

\newcommand{\mc}{{\mathbb C}}

\newcommand{\mh}{{\mathbb H}}

\newcommand{\mU}{\mathbb U}

\newcommand{\eps}{\varepsilon}

\newcommand{\dist}{\operatorname{dist}}

\newcommand{\Complex}{\mathbb{C}}

\newcommand{\todo}[1]{{\sffamily To do:}}

\newtheorem{theorem}{Theorem}

\newtheorem {corollary}{Corollary}
\newtheorem {lemma}{Lemma}

\theoremstyle{definition}
\newtheorem*{rem}{Remark}

\begin{document}

\title{On the discrete spectrum of non-selfadjoint operators}

\author{Michael Demuth\footnotemark[1], Marcel Hansmann\footnotemark[1],
Guy Katriel{\footnotemark[1]\;\;\footnotemark[2]\;\;}}

\date{}

\maketitle
\renewcommand{\thefootnote}{\fnsymbol{footnote}}
\footnotetext[1]{Institute of Mathematics, Technical University of
Clausthal, 38678 Clausthal-Zellerfeld, Germany.}
\footnotetext[2]{Partially supported by the Humboldt Foundation
(Germany).}

\begin{abstract}
We prove quantitative bounds on the eigenvalues of non-selfadjoint
unbounded operators obtained from selfadjoint operators by a
perturbation that is relatively-Schatten. These bounds are applied
to obtain new results on the distribution of eigenvalues of
Schr\"odinger operators with complex potentials.
\end{abstract}

\section{Introduction and Results}
This paper is devoted to the study of the set of isolated
eigenvalues of non-selfadjoint unbounded operators acting on a
Hilbert space $\mathcal{H}$ (throughout this paper all Hilbert spaces are assumed to be complex and separable). More precisely, we are
interested in operators of the form
$$H=H_0+M, \quad \operatorname{dom}(H)=\operatorname{dom}(H_0),$$
where $H_0$ is selfadjoint with spectrum $\sigma(H_0)=[0,\infty),$
and $M$ is a relatively-compact perturbation of $H_0$, i.e.
$\operatorname{dom}(H_0) \subset \operatorname{dom}(M)$ and
$M[\lambda-H_0]^{-1}$ is compact for some (hence all) $\lambda\in
\Complex\setminus [0,\infty)$. Under these assumptions, the
spectrum of $H$  is included in a half-plane and we assume that the same is true for its numerical range $N(H)$,
{\it{i.e.}} there exists an $\omega_0 \geq 0$ such that
\begin{equation*}
\sigma(H) \subset \overline{N(H)} \subset \{ \lambda \in \mc : \Re(\lambda) \geq -\omega_0\}=:
\mh_{\omega_0}.
\end{equation*}
By Weyl's theorem, the essential spectra of $H$ and $H_0$ coincide,
so that $\sigma_{\text{ess}}(H)=[0,\infty)$. However, the
perturbation $M$ may give rise to a discrete set of eigenvalues,
whose only possible limiting points are on the interval $[0,\infty)$
or at infinity. This set of eigenvalues is called the discrete
spectrum of $H$ and will be denoted by $\sigma_{d}(H)$.

It is the aim of this work to obtain further information on
$\sigma_d(H)$, of a quantitative nature, by imposing additional
restrictions on the perturbation $M$. We shall assume, first of all,
that for $p>0$, $M$ is relatively $p$-Schatten, that is, for some
(hence all) $\lambda \in \mc \setminus [0,\infty)$
\begin{equation}\label{aa1}M[\lambda-H_0]^{-1}\in
{\mathbf{S}}_p,\end{equation} where $\mathbf{S}_p$ is the Schatten
class of order $p$ (see Subsection \ref{schatten} for the relevant
definitions). Moreover, we assume that the $p$-Schatten norm
of $M[\lambda-H_0]^{-1}$ satisfies a certain bound (see (\ref{be2})
below), which in particular restricts its growth as $\lambda$
approaches the spectrum of $H_0$, and as $|\lambda|$ goes to
infinity. It will be shown that these assumptions can be used to
derive quantitative information on the set of eigenvalues, in
particular on how \textit{fast} sequences in $\sigma_{d}(H)$ must
converge to $[0,\infty)$. Our main abstract result is presented and
discussed in Subsection \ref{abstract} below.

In Subsection \ref{schr} we apply our
abstract theorem to Schr\"odinger operators with a
complex potential. In this way we demonstrate both that the
hypotheses of our theorem are natural (in the sense that
they can be verified in concrete cases), and that it
yields new results for a problem which has previously been
studied by other methods.

Our abstract result will be proved by constructing a holomorphic
function whose zeros are the eigenvalues of $H$, and using complex
analysis to obtain information on these zeros.
Variants of this approach were used previously, e.g. in
\cite{borichev,demuth}. One of our main tools will be a result of
Borichev, Golinskii and Kupin \cite{borichev} providing bounds on
the zeros of a holomorphic function in the unit disk, in terms of
its growth near the boundary (see also \cite{favorov}).
We note that theorems of this type are a classical theme in complex function
theory, see e.g. \cite[Duren et al.]{duren}, but most results in
this direction, unlike those of \cite{borichev}, are not suitable
for dealing with the kind of holomorphic functions that arise here,
which grow exponentially near the boundary. In \cite{borichev} the complex analysis result was used to obtain inequalities for
eigenvalues of Jacobi operators.

\subsection{Eigenvalue inequalities for general operators}
\label{abstract}

The following theorem is our main result. It will be proved in Section
\ref{proof}.

\begin{theorem}\label{new1}
Let $H_0$ be a selfadjoint operator on a Hilbert space
${\mathcal{H}}$, with $\sigma(H_0)=[0,\infty)$, $H=H_0+M$, where $M$
satisfies (\ref{aa1}) for some $p > 0$, and $N(H) \subset
\mh_{\omega_0}$. Assume that for $\mu\in \Complex$ with
$\Im(\mu)>0$,
\begin{equation}\label{be2}
\|M[\mu^2-H_0]^{-1}\|_{\mathbf{S}_p}^p\leq K_0
\frac{|\mu+i|^{\delta}}{|\Im(\mu)|^{\alpha}|\mu|^\nu},
\end{equation}
where $\alpha, \delta, \nu, K_0 \geq 0$. Let $0 < \tau < 1 $, and
define
\begin{equation}\label{variables}
  \begin{array}{cll}
    \rho & = & \delta +2(p-\alpha) - \nu \\ [4pt]
    \eta_1 &=& \frac{1}{2}(\alpha + 1 + \tau) \\ [4pt]
    \eta_2 &=& \frac{1}{2}(\nu-1+\tau)_+ \\[4pt]
    \eta_3 &=& \frac{1}{2}(\alpha+\nu-\delta)-\tau,
  \end{array}
\end{equation}
where $x_+=\max(x,0)$. Then the following holds,
\begin{equation}
\sum_{\lambda\in
  \: \sigma_{d}(H)}\frac{\dist(\lambda,[0,\infty))^{2\eta_1}}{|\lambda|^{\eta_1 - \eta_2}
  (|\lambda|+1)^{\eta_1 + \eta_2 - \eta_3}}
\leq   C K_0 , \label{result1}
\end{equation}
where $ C=
C(p,\alpha,\delta,\nu,\tau)(1+\omega_0)^{\eta_1+\eta_2+\frac{1}{2}(\alpha+\rho)}$.
\end{theorem}
In the summation over the eigenvalues in (\ref{result1}) and
elsewhere in this article, each eigenvalue is summed according to
its algebraic multiplicity. The constants used in the inequalities
throughout this article will be regarded as generic, i.e. the value
of a constant may change from line to line. However, we will always
carefully indicate the parameters that a constant depends on.

Noting that, for any $\eps>0$ and $|\lambda| \geq \eps$
$$\frac{1}{|\lambda|+1}=\frac{1}{|\lambda|}\frac{1}{1+|\lambda|^{-1}}
\geq \frac{1}{|\lambda|}\frac{1}{1+\eps^{-1}},$$ we obtain the
following corollary of Theorem \ref{new1}.

\begin{corollary}\label{cor1}
Under the assumptions of Theorem \ref{new1} we have, for all
$\eps>0$,
\begin{equation}
\sum_{\lambda\in
  \: \sigma_{d}(H), |\lambda| \geq \eps}\frac{\dist(\lambda,[0,\infty))^{2\eta_1}}{|\lambda|^{2\eta_1 - \eta_3}}
\leq  C K_0 \Big(1+\frac{1}{\eps}\Big)^{\eta_1 + \eta_2 - \eta_3}
\label{result2}
\end{equation}
where $ C=
C(p,\alpha,\delta,\nu,\tau)(1+\omega_0)^{\eta_1+\eta_2+\frac{1}{2}(\alpha+\rho)}$.
\end{corollary}

To demonstrate that these inequalities contain a lot of information
about the discrete spectrum of $H$, we would like to discuss some of
their immediate consequences.

The finiteness of the sum on the LHS of (\ref{result1}) has
consequences regarding sequences $\{\lambda_k\}$  of isolated
eigenvalues converging to some $\lambda^* \in
\sigma_{\text{ess}}(H)$. Taking a subsequence, we can suppose that
one of the following options holds:
\begin{eqnarray*}
\begin{array}{cl}
\mbox{(i)} &  \lambda^*=0 \text{ and } \Re(\lambda_k)\leq 0 \text{ for all } k. \\[4pt]
\mbox{(ii)} &  \lambda^*=0 \text{ and } \Re(\lambda_k)> 0 \text{ for all } k. \\[4pt]
\mbox{(iii)} & \lambda^* \in (0,\infty).
\end{array}
\end{eqnarray*}
In case (i), since $\dist(\lambda_k,[0,\infty))=|\lambda_k|$,
(\ref{result1}) implies
\begin{equation*}\label{nt1}\sum_{k=1}^{\infty}
|\lambda_k|^{\eta_1+\eta_2}<\infty,\end{equation*} which means that
any such sequence must converge to $0$ at a sufficiently fast rate.
In case (ii), (\ref{result1}) implies
\begin{equation*}\label{nt}\sum_{k=1}^{\infty}
\frac{|\Im(\lambda_k)|^{2\eta_1}}{|\lambda_k|^{\eta_1-\eta_2}}<\infty,\end{equation*}

and in case (iii) we obtain
$$\sum_{k=1}^{\infty}
|\Im(\lambda_k)|^{2\eta_1}<\infty,$$ so that the sequence must
converge to the real line sufficiently fast.

Theorem \ref{new1} also provides information about divergent
sequences of eigenvalues. For example, if $\{\lambda_k\}$ is a
sequence of eigenvalues which stays bounded away from $[0,\infty)$,
that is
\begin{equation}\label{qq}
\dist(\lambda_k,[0,\infty))\geq \delta,\end{equation}
 for some $\delta>0$
and all $k$, then (\ref{result2}) implies that
\begin{equation}\label{cd}\sum_{k=1}^{\infty}
\frac{1}{|\lambda_k|^{2\eta_1-\eta_3}}<\infty.\end{equation} In case
$2\eta_1>\eta_3$, (\ref{cd}) implies that $|\lambda_k|$ must go to
infinity sufficiently fast and in case $2\eta_1\leq \eta_3$, (\ref{cd})
implies by contradiction that the number of
eigenvalues outside any $\delta$-neighbourhood of $[0,\infty)$ must
be finite.

For all these results on the asymptotic behaviour of sequences of
eigenvalues, it would be of interest to know whether they are sharp,
that is, if possible, to construct examples of operators that have
precisely the types of asymptotic behaviour indicated above, but no
better.

\subsection{Applications to Schr\"odinger operators}
\label{schr}

We consider Schr\"odinger operators on $L^2(\mr^d)$ that is,
$$H_0 = - \Delta, \qquad H = H_0 + M_V,$$
where $M_V$ is the operator of multiplication with a complex-valued
potential $V$.

Although the study of such operators  has recently attracted
increasing attention (see the papers by Bruneau et al
\cite{bruneau}, Abramov et al. \cite{abramov}, Davies
\cite{Davies02}, Frank et al. \cite{frank} and the monograph by
Davies \cite{Davies07}), relatively little is known in comparison to
the case of real-valued potentials. What is known indicates some
essential differences in the behaviour of the discrete spectrum in
the real and complex cases. For example, while in the real case the
condition $|V(x)|=O((1+|x|)^{-2-\epsilon})$ is sufficient to
guarantee that the number of eigenvalues is finite, Pavlov
\cite{pavlov} (in dimension one) has constructed complex potentials with
$|V(x)|=O(e^{-c|x|^{\alpha}})$, where $\alpha<\frac{1}{2}$, for which
there exists an infinite sequence of eigenvalues, converging to points
in $(0,\infty)$.

For $f \in L^2$ we set $\langle f , f \rangle = \| f \|_{L^2}^2$. The following theorem will be proved in Section \ref{schroedinger}.
\begin{theorem}\label{new2}
Let $H = H_0 + M_V$ acting on $L^2(\mr^d)$, where $d \geq 2$.
Suppose that $V \in L^{p}(\mr^d)$, where $p > \frac{d}{2}$ and $p\geq 2$, and let $\omega_0$ be such that
$$ \langle H_0f,f \rangle + \langle \Re(V)f,f \rangle \geq - \omega_0 \langle f ,f\rangle, \quad f\in \operatorname{dom}(H_0).$$
Then, for any $\tau \in (0,1)$, the following holds: if $p-\frac{d}{2} \geq 1 - \tau$, then
\begin{equation}\label{eq:1}
  \sum_{\lambda\in \:
  \sigma_{d}(H)} \frac{\dist(\lambda,[0,\infty))^{p + \tau}}
  {|\lambda|^{\frac d 4 + \frac 1 2}(|\lambda|+1)^{\frac d 4 - \frac 1 2 +2\tau}}
\leq  C_1  \int_{\mr^d} |V(y)|^{p} dy,
\end{equation}
where $ C_1= C(d,p,\tau)(1+\omega_0)^{(\frac d 4  + p -\frac 1 2 + \tau)}$, and if $p-\frac{d}{2} < 1-\tau$, then
\begin{equation}\label{eq:2}
  \sum_{\lambda\in \:
  \sigma_{d}(H)} \frac{\dist(\lambda,[0,\infty))^{p + \tau}}
  {|\lambda|^{\frac{1}{2}(p + \tau)}(|\lambda|+1)^{\frac 1 2 (d - p + 3 \tau)}}
\leq  C_2  \int_{\mr^d} |V(y)|^{p} dy,
\end{equation}
where $ C_2= C(d,p,\tau)(1+\omega_0)^{\frac 1 2 (d + p  + \tau)}$.
\end{theorem}
Using the same estimate as in the derivation of Corollary \ref{cor1}, the previous theorem implies the following result.
\begin{corollary}\label{cor:2}
Given the assumptions of Theorem \ref{new2}, for any $\tau \in (0,1)$ and $\eps > 0$,
\begin{equation}\label{i3}
  \sum_{\lambda\in \:
  \sigma_{d}(H), |\lambda|\geq \eps}\frac{\dist(\lambda,[0,\infty))^{p + \tau}}
  {|\lambda|^{\frac d 2 + 2\tau}}
\leq  C \Big(1+\frac{1}{\eps}\Big)^{\gamma} \int_{\mr^d}
|V(y)|^{p} dy,
\end{equation}
where $$\gamma=\frac{1}{2}\Big(d-p+3\tau + \Big(p-\frac{d}{2} - 1
+ \tau\Big)_+\Big)$$ and  $$ C=
C(d,p,\tau)(1+\omega_0)^{\frac{1}{2}(d + p+ \tau +( p-\frac{d}{2} - 1 + \tau )_+)}.$$
\end{corollary}

Let us mention that we can derive similar results for dimension
$d=1$, but since some details of the proof are slightly different,
it is not presented here.\\

It is interesting to compare the above estimates with the following
known result for eigenvalues {\it{outside}} of a sector $\{ \lambda:
|\Im(\lambda)| < \chi \: \Re(\lambda)\}$, where $\chi > 0$.
\begin{theorem}[Frank et al. \cite{frank}]\label{thm:frank}
  Let $H=H_0+M_V$ acting on $L^2(\mr^d)$, where $d \geq 1$. Suppose that $V \in L^{\frac d 2 + \kappa}(\mr^d)$ with $\kappa \geq 1$. Then, for any $\chi > 0$,
  \begin{equation}\label{frank}
      \sum_{\lambda \in \; \sigma_{d}(H), |\Im(\lambda)| \geq \chi \Re(\lambda)} |\lambda|^\kappa \leq C(d,\kappa)\left(1+\frac 2 \chi\right)^{\frac d 2 + \kappa} \int_{\mr^d} |V(y)|^{\frac d 2 + \kappa} dy.
  \end{equation}
\end{theorem}
These generalized Lieb-Thirring inequalities were proved by
reduction to a selfadjoint problem, and employing the selfadjoint
Lieb-Thirring inequalites (a similar approach has been used
in \cite{bruneau}). The authors of \cite{frank} conjecture that the
restriction $\kappa \geq 1$ is superfluous, and that (\ref{frank})
might be true for $\kappa$ fulfilling the same restrictions as in
the selfadjoint case, that is $\kappa \geq 0$ when $d\geq 3$,
$\kappa>0$ when $d=2$, and $\kappa\geq \frac{1}{2}$ when $d=1$.

Since the sum in (\ref{frank}) excludes a sector containing the
positive real axis, Theorem \ref{thm:frank} does not provide explicit information on
sequences of eigenvalues converging to some point in $(0,\infty)$,
as is provided in Theorem \ref{new2}. However, the
following corollary of Theorem \ref{thm:frank}, which will be proved
in Section \ref{schroedinger}, gives a bound on a sum over
{\it{all}} eigenvalues. This corollary is of interest in itself, and
also allows a more direct comparison with Theorem \ref{new2}.
\begin{corollary}\label{cor:frank}
  Let $H=H_0+M_V$ acting on $L^2(\mr^d)$, where $d \geq 1$. Suppose that $V \in L^{p}(\mr^d)$
   with $p-\frac{d}{2} \geq 1$. Then, for any $0<\tau<1$,
  \begin{equation}\label{eq_cor_frank}
      \sum_{\lambda \in \: \sigma_d(H)} \frac{\dist(\lambda,[0,\infty))^{p + \tau}}{|\lambda|^{\frac d 2 + \tau}} \leq C(d,p, \tau) \int_{\mr^d} |V(y)|^{p} dy.
  \end{equation}
\end{corollary}

The similarity between the estimates (\ref{eq:1}), (\ref{eq:2}) and (\ref{eq_cor_frank})
is apparent. In particular, as can be seen by comparing Corollary \ref{thm:frank} with Corollary \ref{cor:2}, for eigenvalues accumulating on $(0,\infty)$, Theorem \ref{new2} and Corollary \ref{thm:frank} provide exactly the same estimates. However, for eigenvalues accumulating at $0$, the estimates provided in (\ref{eq_cor_frank}) are stronger than the corresponding estimates provided in (\ref{eq:1}) and (\ref{eq:2}). On the other hand, whereas
(\ref{eq_cor_frank}) has been proved only for $p-\frac{d}{2} \geq 1$, the
inequalities (\ref{eq:1}) and (\ref{eq:2}) remain true as long as $p-\frac{d}{2}
> (2- \frac d 2)_+$, e.g. for any $p > \frac{d}{2}$ when $d \geq 4$.
Thus our method allows to prove a somewhat weaker inequality, which is
valid for a wider range of potentials. Whether this trade-off is an
essential feature of the problem (indicating a different behaviour of the discrete spectrum at $0$ and $\infty$, respectively), or whether it is just an artefact of the
methods used in the proofs of \cite{frank} and in our proof of
Theorem \ref{new2} is an open question, related to
the conjecture made in \cite{frank} which was mentioned above.

\begin{rem}
For another recent generalisation of Theorem \ref{thm:frank} see also \cite{Laptev}.
\end{rem}

\section{Preliminaries}
\label{prem} In this section we gather some results about
determinants and holomorphic functions needed for the proof of
Theorem \ref{new1}. Also, some useful inequalities are derived.
\subsection{Schatten classes and determinants}
\label{schatten}

For a Hilbert space $\mathcal{H}$ let
$\mathbf{C}(\mathcal{H})$ and $\mathbf{B}(\mathcal{H})$ denote the
classes of closed and of bounded linear operators on $\mathcal{H}$,
respectively. We denote the ideal of all compact operators on
$\mathcal{H}$ by $\mathbf{S}_\infty$ and the ideal of all Schatten
class operators by $\mathbf{S}_p, p > 0$, i.e. a compact operator $C
\in \mathbf{S}_p$ if
\begin{equation*}
  \| C\|_{\mathbf{S}_p}^p = \sum_{n=1}^\infty \mu_n(C)^p < \infty
\end{equation*}
where $\mu_n(C)$ denotes the $n$-th singular value of $C$. For $C
\in \mathbf{S}_n$, $n\in \mn$, one can define the (regularized)
determinant
$${\det}_n(I-C)=\prod_{\lambda \in \:\sigma(C)} \left[(1-\lambda) \exp \left( \sum_{j=1}^{n-1} \frac{\lambda^j}{j} \right)\right],$$
having the following properties (see e.g. Dunford et al.
\cite{Dunford}, Gohberg et al. \cite{Gohberg} or Simon
\cite{simonb}):

\noindent 1. $I-C$ is invertible if and only if $\det_n(I-C)\neq 0$.

\noindent 2. $\det_n(I)=1$.

\noindent 3. $\det_n(I-AB)=\det_n(I-BA)$ for $A,B \in
\mathbf{B}(\mathcal{H})$ with $AB, BA \in \mathbf{S}_n$.

\noindent 4. If $C(\lambda)\in \mathbf{S}_n$ depends holomorphically
on $\lambda \in \Omega$, where $\Omega\subset \mc$ is open, then
$\det_n(I-C(\lambda))$ is holomorphic on $\Omega$.

\noindent 5. If $C\in \mathbf{S}_p$ for some $p>0$, then $C\in
\mathbf{S}_{\lceil p \rceil}$, where
\[ \lceil p \rceil = \min \{ n \in \mn :  n \geq p \},\]
and the following inequality holds,
\begin{equation}\label{inequality}
|{\det}_{\lceil p \rceil}(I-C)|\leq \exp \left( \Gamma_p \|C
\|_{\mathbf{S}_p}^p\right),
\end{equation}
where $\Gamma_p $ is some positive constant, see \cite[page
1106]{Dunford}. We remark that $\Gamma_p= \frac{1}{p}$ for $p \leq
1$, $\Gamma_2=\frac 1 2$ and $\Gamma_p \leq e(2+ \log p)$ in
general, see Simon \cite{Simon77}.\\

For $A,B \in \mathbf{B}(\mathcal{H})$ with $B-A \in \mathbf{S}_p$,
the $\lceil p \rceil$-regularized perturbation determinant of $B$ with respect to
$A$ is a well defined holomorphic function on $\rho(A)=\mc \setminus
\sigma(A)$, given by
$$ d(\lambda)= {\det}_{\lceil p \rceil}(I-(\lambda-A)^{-1}(B-A)).$$
Furthermore, $\lambda_0 \in \rho(A)$ is an eigenvalue of $B$ of
algebraic multiplicity $k_0$ if and only if $\lambda_0$ is a zero of
$d(\cdot)$ of the same multiplicity.
\subsection{A theorem of Borichev, Golinskii and Kupin}

The following result, proved in \cite{borichev}, gives a bound on
the zeros of a holomorphic function in the unit disk in terms of its
growth near the boundary. An important feature of this theorem is
that it enables to take into account the existence of `special'
points ($\xi_j$) on the boundary of the unit disk, where the rate
of growth is higher than at generic points.

\begin{theorem}\label{borichev}
Let $h$ be a holomorphic function in the unit disk $\mU$ with
$h(0)=1$. Assume that $h$ satisfies a bound of the form
$$\log|h(z)|\leq K\frac{1}{(1-|z|)^{\alpha}} \prod_{j=1}^N \frac{1}{|z-\xi_j|^{\beta_j}},$$
where $|\xi_j|=1$ ($1\leq j\leq N$), and the exponents
$\alpha,\beta_j$ are nonnegative. Let $\tau>0$. Then the zeros of
$h$ satisfy the inequality
\begin{equation*}\label{sc2}\sum_{h(z)=0}(1-|z|)^{\alpha+1+\tau}\prod_{j=1}^N|z-\xi_j|^{(\beta_j-1+\tau)_+}\leq
C(\alpha,\{ \beta_j\},\{\xi_j\}, \tau)K.\end{equation*}
\end{theorem}

\subsection{Some inequalities}\label{sec:ineq}

We will need some elementary inequalities, which we collect
here for convenient reference.

\begin{lemma}\label{sq} For $\mu\in \Complex$
$$|\mu||\Im(\mu)|\leq \dist(\mu^2,[0,\infty))\leq 2|\mu| |\Im(\mu)|.$$
\end{lemma}

\begin{proof}
If $\Re(\mu^2)>0$ then $|\Re(\mu)|>|\Im(\mu)|$ and we have
$$\dist(\mu^2,[0,\infty))=|\Im(\mu^2)|=2|\Re(\mu)||\Im(\mu)|\leq 2|\mu||\Im(\mu)|,$$
$$\dist(\mu^2,[0,\infty))=|\Im(\mu^2)|=2|\Re(\mu)||\Im(\mu)|\geq \sqrt{2}|\mu||\Im(\mu)|.$$
If $\Re(\mu^2)\leq 0$ then $|\Re(\mu)|\leq |\Im(\mu)|$ and we have
$$\dist(\mu^2,[0,\infty))=|\mu|^2=\Re(\mu)^2+\Im(\mu)^2\leq 2\Im(\mu)^2\leq 2|\Im(\mu)||\mu|,$$
$$\dist(\mu^2,[0,\infty))=|\mu|^2\geq |\mu||\Im(\mu)|.$$
Taking the worst-case scenarios we get the result.
\end{proof}

For $a>0$, we define the conformal map $\phi_a :\mU\rightarrow
\Complex\setminus [0,\infty)$ by
\begin{equation}\label{l}\phi_a(z)=-a^2\left(\frac{z+1}{z-1}\right)^2.\end{equation}
Note that $\phi_a(0)=-a^2$.

\begin{lemma}\label{ese2} For $a>0$ and $\lambda\in \Complex\setminus
[0,\infty)$, the following holds
\begin{equation*}
\frac a 2
\frac{\dist(\lambda,[0,\infty))}{|\lambda|^{\frac{1}{2}}(|\lambda|+a^2)}\leq
1-|\phi_a^{-1}(\lambda)| \leq 4a
\frac{\dist(\lambda,[0,\infty))}{|\lambda|^{\frac{1}{2}}(|\lambda|+a^2)},
\end{equation*}
\begin{equation*}
\frac{\sqrt{2}a}{(|\lambda|+a^2)^{\frac{1}{2}}}\leq
|\phi_a^{-1}(\lambda)-1|\leq
\frac{2a}{(|\lambda|+a^2)^{\frac{1}{2}}}, \end{equation*}
\begin{equation*}
\frac{\sqrt{2}|\lambda
|^{\frac{1}{2}}}{(|\lambda|+a^2)^{\frac{1}{2}}}\leq
|\phi_a^{-1}(\lambda)+1|\leq \frac{2|\lambda
|^{\frac{1}{2}}}{(|\lambda|+a^2)^{\frac{1}{2}}}. \end{equation*}

\end{lemma}

\begin{proof}
This is a standard computation, see e.g. Corollary 1.4 in \cite{Pommerenke}.
\end{proof}

\section{Proof of the eigenvalue inequalities for general operators}\label{proof}

The following Lemma, which is of independent interest, is the main
ingredient in the proof of Theorem \ref{new1}, provided in Section
\ref{proof_ab}.
\begin{lemma} \label{mainlemma}
Let $H_0$ be selfadjoint with $\sigma(H_0)=[0,\infty)$, $H=H_0+M$,
where $M$ satisfies (\ref{aa1}) for some $p > 0$, and $N(H) \subset
\mh_{\omega_0}$. For $a>0$ with $a^2>\omega_0$ and $\mu\in \Complex$
with $\Im(\mu)>0$ assume that
\begin{equation}\label{be}
\|[a^2+H]^{-1}M[\mu^2-H_0]^{-1}\|_{\mathbf{S}_p}^p\leq K_1
\frac{|\mu+ia|^{\delta}}{|\Im(\mu)|^{\alpha}|\mu|^\nu},
\end{equation}
where $\alpha, \delta, \nu \geq 0$ and $K_1 > 0$. For $0 < \tau < 1 $
let $\rho, \eta_1, \eta_2$ be defined as in (\ref{variables}), and
let
\begin{equation}
\eta_0=\frac{1}{2}(\rho -1 + \tau)_+.\label{eq:mainlemma}
\end{equation}
Then
\begin{eqnarray}\label{ff4}
\sum_{\lambda\in
\sigma_{d}(H)}\frac{\dist(\lambda,[0,\infty))^{2\eta_1}}{|\lambda|^{\eta_1
- \eta_2}(a^2+ |\lambda|)^{\eta_0+2\eta_1 + \eta_2} } &\leq&
C(p,\alpha,\delta,\nu,\tau) \frac{ K_1}{
a^{2\eta_0+2\eta_1-\alpha-\rho}}.
\end{eqnarray}
\end{lemma}

\subsection{Proof of Lemma \ref{mainlemma}}

We start with the construction of a holomorphic function
$f:\Complex\setminus [0,\infty) \rightarrow \Complex$ whose zeros
coincide with the eigenvalues of $H$ in $\mc \setminus [0,\infty)$.

\subsubsection{The function $f(\lambda)$}

To begin with, we note that the resolvent-identity
$$[a^2+H_0]^{-1}-[a^2+H]^{-1}=[a^2+H]^{-1}M[a^2+H_0]^{-1}$$
implies that
\begin{eqnarray} &&I-(\lambda+a^2)[a^2+H]^{-1} \nonumber \\
&=& I-(\lambda +
a^2)[a^2+H_0]^{-1}+(\lambda+a^2)[a^2+H]^{-1}M[a^2+H_0]^{-1}.
\label{e1}
\end{eqnarray}
Assuming that $\lambda \notin [0,\infty)$ and multiplying both sides
of (\ref{e1}) from the right by
$[I-(\lambda+a^2)[a^2+H_0]^{-1}]^{-1}$ we obtain
\begin{eqnarray}\label{a12}&&\left[I-(\lambda+a^2)[a^2+H]^{-1}\right] \left[I-(\lambda+a^2)[a^2+H_0]^{-1}\right]^{-1} \\
&=&I+(\lambda+a^2)[a^2+H]^{-1}M[a^2+H_0]^{-1} \left[I-(\lambda+a^2)[a^2+H_0]^{-1}\right]^{-1} \nonumber\\
&=&I-(\lambda+a^2)[a^2+H]^{-1}M[\lambda-H_0]^{-1}.\nonumber\end{eqnarray}
Note that the LHS of (\ref{a12}) is invertible if and only if
$I-(\lambda+a^2)[a^2+H]^{-1}$ is invertible, which is the case if
and only if $\lambda \not \in \; \sigma_{d}(H)$. Therefore, defining
\begin{equation}\label{F}F(\lambda)=(\lambda+a^2)[a^2+H]^{-1}M[\lambda-H_0]^{-1},\end{equation} it
follows that \begin{equation}\label{dd}\lambda\in
\sigma_{d}(H)\;\;\Leftrightarrow \;\;I-F(\lambda) {\mbox{  is not
invertible}}.\end{equation} $F(\lambda)$ is an operator-valued
function defined on $\Complex\setminus [0,\infty)$, and by
assumption (\ref{aa1}) we have $F(\lambda)\in {\mathbf{S}}_p$.
Hence, (\ref{dd}) can be rewritten as
$$\lambda\in
\sigma_{d}(H) \quad \Leftrightarrow \quad {\det}_{\lceil p
\rceil}(I-F(\lambda))=0.$$
Defining $f(\lambda)={\det}_{\lceil p \rceil}(I-F(\lambda))$, we
obtain that $f$ is holomorphic on $\Complex\setminus [0,\infty)$ and
\begin{equation}\label{z}
\sigma_{d}(H)=\{ \lambda\in \Complex\setminus [0,\infty)\;|\;
f(\lambda)=0 \}.
\end{equation}
Moreover, $F(-a^2)=0$ implies that
\begin{equation}\label{o}f(-a^2)=1.\end{equation}
It should be noted that
$$F(\lambda)= [(\lambda+ a^2)^{-1}-[a^2+H_0]^{-1}]^{-1}([a^2+H]^{-1}-[a^2+H_0]^{-1}),$$
providing the alternative representation
$$f(\lambda) = {\det}_{\lceil p \rceil}(I-[(\lambda+ a^2)^{-1}-[a^2+H_0]^{-1}]^{-1}([a^2+H]^{-1}-[a^2+H_0]^{-1})).$$
This shows that $f$ is the $\lceil p \rceil$-regularized perturbation
determinant of $[a^2+ H]^{-1}$ with respect to $[a^2+H_0]^{-1}$ as
defined in Subsection \ref{schatten}. Together with the spectral
mapping theorem this implies that the order of $\lambda_0$ as a zero
of $f$ coincides with its algebraic multiplicity as an eigenvalue of
$H$.\\

We conclude this subsection  with the following bound on
$f(\lambda)$.

\begin{lemma}\label{bound} Assume (\ref{be}).
Then for all $\mu\in \Complex$ with $\Im(\mu)>0$
\begin{equation}\label{ee11}\log |f(\mu^2)|\leq \Gamma_p K_1
\frac{ |\mu-ia|^p |\mu+ia|^{\delta+p}}{|\Im(\mu)|^{\alpha}|\mu|^\nu}
\end{equation}
\end{lemma}

\begin{proof}
Using (\ref{F}) and (\ref{be}), we have
$$\|F(\mu^2)\|_{\mathbf{S}_p}^p\leq |\mu^2+a^2|^p\|[a^2+H]^{-1}M[\mu^2-H_0]^{-1}\|_{\mathbf{S}_p}^p
\leq   K_1
\frac{|\mu^2+a^2|^p|\mu+ia|^{\delta}}{|\Im(\mu)|^{\alpha}|\mu|^\nu}
,$$ and the result follows by (\ref{inequality}).
\end{proof}

In the sequel, we want to study the zeros of $f(\lambda)$. Since our
tool will be a theorem on zeros of a holomorphic function in the
unit disk $\mU$, we have to transform the problem from
$\Complex\setminus [0,\infty)$ to $\mU$.

\subsubsection{The function $h(z)$}

Recall the conformal map $\phi_a :\mU\rightarrow \Complex\setminus
[0,\infty)$ given by (\ref{l}), and define $h:\mU\rightarrow
\Complex$ by
$$h(z)=f(\phi_a(z)).$$
Then $h$ is holomorphic in the unit disk, and (\ref{z}) implies that
\begin{equation}\label{si}\sigma_{d}(H)=\{\phi_a(z)\;|\; z\in \mU,\;h(z)=0\}.\end{equation} By (\ref{o}) we
have
\begin{equation*}\label{o1} h(0)=1.\end{equation*}
The bound on $f$ provided by Lemma \ref{bound} is now translated
into a bound on $h$.

\begin{lemma}\label{bh} Assume (\ref{be}). Then for all $z\in \mU$
\begin{equation*}\label{ee2}\log|h(z)|\leq C(p,\delta) K_1 a^{\alpha+\rho}
 \frac{|z|^{p}}{(1-|z|)^{\alpha}|z+1|^\nu|z-1|^{\rho}},\end{equation*} where
$\rho$ was defined in (\ref{variables}).
\end{lemma}

\begin{proof}
Set $\mu=ia\frac{1+z}{1-z}$ and note that $\Im(\mu)>0$. Then by
Lemma \ref{bound} \begin{equation}\label{io}\log|h(z)|=\log
|f(\mu^2)|\leq \Gamma_p K_1
\frac{|\mu-ia|^p|\mu+ia|^{\delta+p}}{|\Im(\mu)|^{\alpha}|\mu|^\nu}.\end{equation}
Since
$$|\mu+ia|=\frac{2a}{|z-1|},\;\; |\mu-ia|=\frac{2a|z|}{|z-1|}\;\; \text{ and } \;\; \frac{1}{|\Im(\mu)|} \leq \frac{|1-z|^2}{a(1-|z|)},$$


\noindent we obtain
 $$\frac{|\mu-ia|^p|\mu+ia|^{\delta+p}}{|\Im(\mu)|^{\alpha}|\mu|^\nu}\leq
 2^{\delta+2p} a^{\delta+2p-\alpha-\nu}
 \frac{|z|^{p}|z-1|^{2\alpha+\nu-\delta-2p}}{(1-|z|)^{\alpha}|z+1|^\nu},$$
 which together with (\ref{io}) concludes the proof.
\end{proof}

We are now in a position to apply Theorem \ref{borichev}, the result
by Borichev, Golinskii and Kupin. Since
$\rho=\delta+2(p-\alpha)-\nu$ can be negative, Lemma \ref{bh}
implies that
$$ \log|h(z)|\leq C(p,\alpha,\delta, \nu) K_1 a^{\alpha+\rho} \frac{|z|^{p}}{(1-|z|)^{\alpha}|z-1|^{\rho_+}|z+1|^\nu}.$$
Applying  Theorem \ref{borichev} with $N=2, \xi_1=1, \xi_2=-1,
\beta_1= \rho_+, \beta_2=\nu$ and $K= C(p,\alpha,\delta, \nu)
K_1a^{\alpha+\rho}$ we obtain, for $0 < \tau < 1$,
\begin{eqnarray}\label{sc3}
&&\sum_{h(z)=0, z \in \mU}
(1-|z|)^{\alpha+1+\tau}|z-1|^{(\rho-1+\tau)_+}|z+1|^{(\nu-1+\tau)_+} \nonumber \\
&\leq& C(p,\alpha,\delta,\nu,\tau)K_1 a^{\alpha+\rho}.
\end{eqnarray}
Here it was used that $(\rho_+-1+\tau)_+=(\rho-1+\tau)_+$ for $0 <
\tau  < 1$.

Recalling the definition of $\eta_0, \eta_1$ and $\eta_2$ (see
(\ref{variables}) and (\ref{eq:mainlemma})),
inequality (\ref{sc3}) can be rewritten as follows
\begin{equation}\label{sc4}
\sum_{h(z)=0, z \in \mU}
(1-|z|)^{2\eta_1}|z+1|^{2\eta_2}|z-1|^{2\eta_0} \leq
C(p,\alpha,\delta,\nu,\tau)K_1 a^{\alpha+\rho}.
\end{equation}

\subsubsection{Back to the eigenvalues}

It remains to retranslate the bound obtained in (\ref{sc4}) into a
bound on the eigenvalues of $H$. Using the inequalities derived in
Subsection \ref{sec:ineq} this is straightforward. From (\ref{si})
and (\ref{sc4}) we obtain
\begin{eqnarray}\label{sc7}
&& \sum_{\lambda\in\:\sigma_{d}(H)}
(1-|\phi_a^{-1}(\lambda)|)^{2\eta_1}
|\phi_a^{-1}(\lambda)+1|^{2\eta_2} |\phi_a^{-1}(\lambda)-1|^{2\eta_0} \\
&\leq& C(p,\alpha,\delta,\nu,\tau)K_1 a^{\alpha+\rho} \nonumber
\end{eqnarray}
and, using Lemma \ref{ese2}, the sum on the left-hand side of (\ref{sc7}) can be bounded from below by
\begin{equation}\label{f2}
C(p,\alpha,\delta,\nu,\tau)a^{2\eta_0 + 2\eta_1} \sum_{\lambda \in \sigma_d(H)}
\frac{\dist(\lambda,[0,\infty))^{2\eta_1}}{|\lambda|^{\eta_1-\eta_2}}
\frac{1}{(|\lambda|+a^2)^{\eta_0+2\eta_1+\eta_2}}.
\end{equation}
(\ref{sc7}) and (\ref{f2})  imply (\ref{ff4}). We have thus completed the proof of
Lemma \ref{mainlemma}.

\subsection{Proof of Theorem \ref{new1}}\label{proof_ab}
We will see that Theorem \ref{new1} is a direct consequence of Lemma
\ref{mainlemma}:\\[4pt]
Let $a > 0$ with $a^2 > \omega_0$. Since
$$\|[a^2+H]^{-1}\| \leq \frac 1 {\dist(-a^2,\overline{N(H)})}\leq \frac{1}{a^2-\omega_0},$$
and for $\Im(\mu)>0$
$$|\mu + i| \leq \left(1 + \frac 1 a \right) |\mu + ia| \leq \sqrt{2 \left(1+\frac 1 {a^2}\right)} | \mu + ia|,$$
we obtain from assumption (\ref{be2}) that
\begin{eqnarray*}
  \|[a^2+H]^{-1}M[\mu^2-H_0]^{-1}\|_{\mathbf{S}_p}^p &\leq&
\frac{2^{\frac \delta 2}K_0
(1+\frac{1}{a^2})^{\frac{\delta}{2}}}{(a^2-\omega_0)^{p}}
\frac{|\mu+ia|^{\delta}}{|\Im(\mu)|^{\alpha}|\mu|^\nu}.
\end{eqnarray*}
Hence, an application of Lemma \ref{mainlemma} shows that
\begin{eqnarray*}
\sum_{\lambda\in
\:\sigma_{d}(H)}\frac{\dist(\lambda,[0,\infty))^{2\eta_1}}{|\lambda|^{\eta_1
- \eta_2}(a^2+ |\lambda|)^{\eta_0 + 2\eta_1 + \eta_2} } &\leq&
L \frac{ (a^2+1)^{\frac{\delta}{2}}}{(a^2-\omega_0)^p a^{2\eta_0
+ 2\eta_1-\alpha-\rho+\delta}} ,\end{eqnarray*} where
$L=C(p,\alpha,\delta,\nu,\tau)K_0$. To simplify the notation we set
\begin{eqnarray*}
  b&=&a^2, \\
  \varphi_1&=&\eta_0 + \eta_1-\frac{\alpha+\rho-\delta}{2}+p-1-\tau, \\[3pt]
  \varphi_2&=&\eta_0 + 2\eta_1 +
\eta_2.
\end{eqnarray*}
Then, the last inequality is equivalent to
\begin{eqnarray}
\sum_{\lambda\in\:
\sigma_{d}(H)}\frac{\dist(\lambda,[0,\infty))^{2\eta_1}b^{\varphi_1}}{|\lambda|^{\eta_1
- \eta_2}(b+ |\lambda|)^{\varphi_2}(b+1)^{\frac{\delta}{2}} } &\leq&
L \frac{ b^{p-1-\tau}}{(b-\omega_0)^p}. \label{pr1}
\end{eqnarray}
Note that (\ref{pr1}) holds for any $b>\omega_0$, so we may
integrate both sides of (\ref{pr1}) with respect to $b \in
(\omega_0+1,\infty)$. For the RHS, we obtain
\begin{equation}\label{pr2}
\int_{\omega_0+1}^\infty db \; \frac{b^{p-1-\tau}}{(b-\omega_0)^{p}}
\leq 
\tau^{-1}(1+\omega_0)^{p-\tau}.
\end{equation}
Integrating the LHS of (\ref{pr1}), interchanging sum and integral,
it follows that
\begin{eqnarray}\label{d1}&&\int_{\omega_0+1}^\infty db \Big[
\sum_{\lambda\in\:
\sigma_{d}(H)}\frac{\dist(\lambda,[0,\infty))^{2\eta_1}b^{\varphi_1}}{|\lambda|^{\eta_1
- \eta_2}(b+ |\lambda|)^{\varphi_2}(b+1)^{\frac{\delta}{2}}
}\Big]\\&=& \sum_{\lambda\in\:
  \sigma_{d}(H)}\frac{\dist(\lambda,[0,\infty))^{2\eta_1}}{|\lambda|^{\eta_1 - \eta_2}} \int_{\omega_0+1}^\infty
  db \; \frac{b^{\varphi_1}}
  {(b+ |\lambda|)^{\varphi_2}(b+1)^{\frac{\delta}{2}}}.
\nonumber
\end{eqnarray}
The finiteness of the above integral follows from (\ref{pr2}), and
we can bound it from below as follows,
\begin{small}
\begin{eqnarray}
\int_{\omega_0+1}^\infty db \; \frac{b^{\varphi_1}}{(b+
|\lambda|)^{\varphi_2}(b+1)^{\frac{\delta}{2}}}
\geq \frac{C(p,\alpha,\delta,\nu,\tau)}{(\omega_0+1)^{\eta_1+\eta_2-\eta_3}(|\lambda|+1)^{\eta_1+\eta_2-\eta_3}}.
\label{pr3}
\end{eqnarray}
\end{small}
Note that we used the easily verified fact that $\eta_1+\eta_2-\eta_3 > 0$, see definition (\ref{variables}) above.
(\ref{pr1}) to (\ref{pr3}) imply that
\begin{equation*}
\sum_{\lambda\in\:
  \sigma_{d}(H)}\frac{\dist(\lambda,[0,\infty))^{2\eta_1}}{|\lambda|^{\eta_1 - \eta_2}(|\lambda|+1)^{\eta_1 + \eta_2 - \eta_3}}
\leq C(p,\alpha,\delta,\nu,\tau) L
(1+\omega_0)^{p-\tau+\eta_1+\eta_2-\eta_3}.
\end{equation*}
Noting that $p-\tau-\eta_3= \frac{\alpha+\rho}{2}$ concludes the
proof of Theorem \ref{new1}.

\section{Proof of the inequalities for Schr\"odinger operators} \label{schroedinger}

\subsection{Schatten norm bounds}

We intend to prove Theorem \ref{new2} by an application of Theorem \ref{new1}. To this end, some
information on the Schatten norms of $M_V[\mu^2-H_0]^{-1}$ is
needed. This will be dealt with in the following two lemmas.

\begin{lemma}\label{lem_es1} Let $V \in L^p(\mr^d)$, where $d\geq 2$, $p\geq 2$ and $p>\frac{d}{2}$. Then, for $\lambda\in \Complex\setminus [0,\infty)$  with $\Re(\lambda) > 0$,
$$\|M_V[\lambda-H_0]^{-1}\|_{\mathbf{S}_p}^p\leq C(p,d) \|V\|_{L^p}^p \left[\frac{|\Re(\lambda)|^{\frac{d-2}{2}}}{|\Im(\lambda)|^{p-1}}
+\frac{1}{|\Im(\lambda)|^{p-\frac{d}{2}}} \right] $$ and for
$\lambda \in \mc$ with $\Re(\lambda) \leq 0$,
$$ \|M_V[\lambda-H_0]^{-1}\|_{\mathbf{S}_p}^p\leq C(p,d) \|V\|_{L^p}^p \frac{1}{|\lambda|^{p-\frac d 2}}.$$
\end{lemma}

\begin{proof}
Theorem 4.1 in Simon \cite{simonb} implies, for $p\geq 2$,
$$\|M_V[\lambda-H_0]^{-1}\|_{\mathbf{S}_p}^p\leq (2\pi)^{-\frac d p}\|(\lambda-|\:.\:|^2)^{-1}\|_{L^p}^p \|V\|_{L^p}^p.$$
We will show that, for $\lambda\in \Complex\setminus [0,\infty)$
with $\Re(\lambda) > 0$
\begin{equation} \|(\lambda-|\:.\:|^2)^{-1}\|_{L^p}^p\leq C(p,d) \left[ \frac{|\Re(\lambda)|^{\frac{d-2}{2}}}{|\Im(\lambda)|^{p-1}}
+\frac{1}{|\Im(\lambda)|^{p-\frac{d}{2}}} \right]. \label{show1}
\end{equation}
Set $\lambda=\lambda_0+i\lambda_1$, and assume first that $\lambda_0
> 0$. Since $\|(\lambda-|\:.\:|^2)^{-1}\|_{L^p}=
\|(\overline{\lambda}-|\:.\:|^2)^{-1}\|_{L^p}$, it is sufficient to
treat the case $\lambda_1 > 0$.
Making the change of variable $r=\sqrt{\lambda_0-\lambda_1 s}$ we can express $\|(\lambda-|\:.\:|^2)^{-1}\|_{L^p}^p
$ as
\begin{small}
\begin{eqnarray}\label{es1}
C(d)\lambda_1^{1-p}\left[\int_{0}^{\infty}\frac{(\lambda_0+\lambda_1
s)^{\frac{d-2}{2}}}{(s^2+1)^{\frac{p}{2}}}
\; ds+\int_{0}^{\frac{\lambda_0}{\lambda_1}}\frac{(\lambda_0-\lambda_1 s)^{\frac{d-2}{2}}}{(s^2+1)^{\frac{p}{2}}}
\; ds\right].
\end{eqnarray}
\end{small}
For the first integral in (\ref{es1}), we have, using
$(\lambda_0+\lambda_1 s)^{\frac{d-2}{2}}\leq
(2\lambda_0)^{\frac{d-2}{2}}+(2\lambda_1 s)^{\frac{d-2}{2}}$,
\begin{eqnarray}\label{kj}
\int_{0}^{\infty}\frac{(\lambda_0+\lambda_1 s)^{\frac{d-2}{2}}}{(s^2+1)^{\frac{p}{2}}}\;ds
&\leq&C(d,p)[\lambda_0^{\frac{d-2}{2}}+\lambda_1^{\frac{d-2}{2}}].
\end{eqnarray}
Similarly, for the second integral in (\ref{es1}) we obtain
\begin{eqnarray}\label{ll}
\int_{0}^{\frac{\lambda_0}{\lambda_1}}\frac{(\lambda_0-\lambda_1 s)^{\frac{d-2}{2}}}{(s^2+1)^{\frac{p}{2}}}\;
ds
&\leq& \lambda_0^{\frac{d-2}{2}}
\int_{0}^{\infty}\frac{1}{(s^2+1)^{\frac{p}{2}}}
ds=C(p)\lambda_0^{\frac{d-2}{2}}.
\end{eqnarray}
(\ref{es1}), (\ref{kj}) and (\ref{ll}) imply the validity of (\ref{show1}) in case that $\lambda_0 > 0$.
A similar argument shows the validity of (\ref{show1}) in case that $\lambda_0 \leq 0$.


\end{proof}

\begin{lemma}\label{po}
Let $V \in L^p(\mr^d)$, where $d \geq 2$, $p \geq 2$ and $p > \frac d 2 $. Then, for $\mu \in \mc$
with $\Im(\mu)>0$,
\begin{equation}
\|M_V[\mu^2-H_0]^{-1}\|_{\mathbf{S}_p}^p\leq C(p,d) \|V\|_{L^p}^p
\frac{|\mu+i|^\delta}{|\Im(\mu)|^\alpha |\mu|^\nu} \label{eq:po}
\end{equation}
where
$$\nu=p-\frac{d}{2}, \quad \delta=\frac{d}{2}-1,\quad \alpha=p-1.$$
\end{lemma}

\begin{proof}
Let us consider first the case $0< \Im(\mu) < |\Re(\mu)|$. Since
$\Re(\mu^2)= \Re(\mu)^2-\Im(\mu)^2> 0$ and $\Im(\mu^2)=2\Re(\mu)\Im(\mu)$, Lemma \ref{lem_es1} implies
$$\|M_V[\mu^2-H_0]^{-1}\|_{\mathbf{S}_p}^p
\leq C(p,d) \|V\|_{L^p}^p  \left[\frac{|\Re(\mu)^2-\Im(\mu)^2|^{\frac{d-2}{2}}}{|2\Re(\mu)\Im(\mu)|^{p-1}}
+\frac{1}{|2\Re(\mu)\Im(\mu)|^{p-\frac{d}{2}}} \right].$$
Hence, to show (\ref{eq:po}), it is sufficient to show that
\begin{equation}\label{s11}
\frac{|\Re(\mu)^2-\Im(\mu)^2|^{\frac{d-2}{2}}
|\mu|^{p-\frac{d}{2}}}{|2\Re(\mu)|^{p-1}|\mu+i|^{\frac{d}{2}-1}}\;\;  \text{ and } \;\; \frac{ |\mu|^{p-\frac{d}{2}} \Im(\mu)^{\frac d 2
-1}}{|2\Re(\mu)|^{p-\frac d 2}|\mu+i|^{\frac{d}{2}-1}}
\end{equation}
are bounded from above by a suitable constant $C(p,d)$. In the following, we will provide such a bound for the first quotient in (\ref{s11}) (a similar computation for the second quotient will be omitted). Since $|\Re(\mu)|>\Im(\mu)>0$ one deduces that
$|\mu|\leq \sqrt{2} |\Re(\mu)|$  and $|\Re(\mu)^2-\Im(\mu)^2| \leq 2 | \Re(\mu) |^2$. Thus, we obtain
\begin{eqnarray*}
\frac{|\Re(\mu)^2-\Im(\mu)^2|^{\frac{d-2}{2}}
|\mu|^{p-\frac{d}{2}}}{|2\Re(\mu)|^{p-1}|\mu+i|^{\frac{d}{2}-1}}
&\leq& \frac 1 {2^{ \frac p 2 - \frac d 4}} \left(
\frac{|\Re(\mu)|}{|\mu+i|}\right)^{\frac{d}{2}-1} \leq \frac 1 {2^{
\frac p 2 - \frac d 4}}.
\end{eqnarray*}
The proof of (\ref{eq:po}) in case that $\Im(\mu) > 0$ and $|\Re(\mu)| \leq \Im(\mu)$ follows the same lines as above and will therefore be omitted.
\end{proof}

\subsection{Proof of Theorem \ref{new2}}
Lemma \ref{po} implies that for $p \geq 2$ and $p>\frac d 2$
\begin{equation*}
\|M_V[\mu^2-H_0]^{-1}\|_{\mathbf{S}_p}^p\leq C(p,d) \| V\|_{L^p}^p
\frac{|\mu+i|^\delta}{|\Im(\mu)|^\alpha |\mu|^\nu},
\end{equation*}
where $\nu=p-\frac{d}{2}, \: \delta=\frac{d}{2}-1$ and $\alpha=p-1$.
With the notation of Theorem \ref{new1} we have for $0<\tau < 1$,
\begin{equation*}
  \begin{array}{cll}
    \rho & = & \delta +2(p-\alpha) - \nu = d - p +1  \\ [4pt]
    \eta_1 &=& \frac{1}{2}(\alpha + 1 + \tau)= \frac 1 2 ( p + \tau ) \\ [4pt]
    \eta_2 &=& \frac{1}{2}(\nu-1+\tau)_+ = \frac 1 2 ( p - \frac d 2 - 1 + \tau )_+ \\[4pt]
    \eta_3 &=& \frac{\alpha+\nu-\delta}{2}-\tau = p - \frac d 2 - \tau
  \end{array}
\end{equation*}
and an application of Theorem \ref{new1} shows that
\begin{equation*}
\sum_{\lambda\in\:
  \sigma_{d}(H)}\frac{\dist(\lambda,[0,\infty))^{p + \tau}}{|\lambda|^{\frac 1 2 (p+\tau-(p-\frac d 2-1 +\tau)_+)}(|\lambda|+1)^{\frac 1 2 (d-p + 3\tau + (p-\frac d 2 -1+\tau)_+)}}
\leq  C \|V\|_{L^p}^p,
\end{equation*}
where $C=C(d,p,\tau)(1+\omega_0)^{\frac 1 2 ( p + d + \tau + (p - \frac d 2 - 1 + \tau)_+)}$. Simplifying 
 the above expression in the cases $p-\frac{d}{2}\geq 1-\tau$ and $p-\frac{d}{2}< 1-\tau$  we get
 (\ref{eq:1}), (\ref{eq:2}).

\subsection{Proof of Corollary \ref{cor:frank}}

Restricting the generalized Lieb-Thirring inequality (\ref{frank})
to the set $\Re(\lambda)>0$, we obtain
  \begin{equation}\label{frank2}
      \sum_{\lambda \in \; \sigma_{d}(H), |\Im(\lambda)| \geq \chi \Re(\lambda)>0} |\lambda|^\kappa \leq C(d,\kappa)\left(1+\frac 2 \chi\right)^{\frac d 2 + \kappa} \int_{\mr^d} |V(y)|^{\frac d 2 + \kappa} dy.
  \end{equation}
We multiply both sides of (\ref{frank2}) with $\chi^{\frac d 2 +
\kappa-1+\tau}$, where $0 < \tau < 1$, and integrate over $\chi \in
(0,1)$. Interchanging sum and integral, one obtains for the LHS
\begin{eqnarray*}
&& \int_0^1 d\chi \; \chi^{\frac d 2 + \kappa-1+\tau} \sum_{\lambda \in \; \sigma_{d}(H), |\Im(\lambda)| \geq \chi \Re(\lambda)>0} |\lambda|^\kappa \\
&=& \sum_{\lambda \in \; \sigma_{d}(H), \Re(\lambda)>0} |\lambda|^\kappa \int_0^{\min(\frac{|\Im(\lambda)|}{\Re(\lambda)},1)} d\chi \; \chi^{\frac d 2 + \kappa-1+\tau} \\
&=& C(d,\kappa,\tau) \sum_{\lambda \in \; \sigma_{d}(H), \Re(\lambda)>0} |\lambda|^\kappa \min\left(1,\left(\frac{|\Im(\lambda)|}{\Re(\lambda)}\right)^{\frac d 2 + \kappa+\tau}\right) \\
&\geq& C(d,\kappa,\tau) \sum_{\lambda \in \; \sigma_{d}(H), |\Im(\lambda)| \leq \Re(\lambda)} |\lambda|^\kappa \left(\frac{|\Im(\lambda)|}{\Re(\lambda)}\right)^{\frac d 2 + \kappa+\tau} \\
&\geq& C(d,\kappa,\tau) \sum_{\lambda \in \; \sigma_{d}(H),
|\Im(\lambda)| \leq \Re(\lambda)}
\frac{\dist(\lambda,[0,\infty))^{\frac d 2 + \kappa
+\tau}}{|\lambda|^{\frac d 2 + \tau}}.
\end{eqnarray*}
Similarly, we obtain for the RHS of (\ref{frank2})
\begin{small}
\begin{eqnarray*}
\int_0^1 d\chi \; \left(1+\frac 2 \chi\right)^{\frac d 2 + \kappa}
\chi^{\frac d 2 + \kappa - 1 + \tau}  \int_{\mr^d} |V(y)|^{\frac d 2
+ \kappa} dy   \leq C(d,\kappa,\tau) \int_{\mr^d} |V(y)|^{\frac d 2
+ \kappa} dy .
\end{eqnarray*}
\end{small}
This shows that
$$    \sum_{\lambda \in \; \sigma_{d}(H), |\Im(\lambda)| \leq \Re(\lambda)} \frac{\dist(\lambda,[0,\infty))^{\frac d 2 + \kappa + \tau}}{|\lambda|^{\frac d 2 + \tau}} \leq C(d,\kappa,\tau) \int_{\mr^d} |V(y)|^{\frac d 2 + \kappa} dy.$$
Using Theorem \ref{thm:frank} with $\chi=1$ gives that the same
inequality is true summing over all eigenvalues $\lambda$ with
$|\Im(\lambda)| \geq \Re(\lambda)$. Setting $p=\kappa+\frac{d}{2}$ completes the proof of
Corollary \ref{cor:frank}.

\end{document}